\RequirePackage{fix-cm}
\documentclass[smallextended]{svjour3}       
\smartqed  
\usepackage{graphicx}
\usepackage{mathptmx}      
\usepackage{mathrsfs}
\usepackage{color}
\usepackage{amssymb}
\usepackage{enumerate}
\newcommand{\sqn}{\sqrt{n}}
\newcommand{\baa}{\bar{a}}
\newcommand{\orp}{\textup{ord}_p}

\newcommand{\ZZ}{{\mathbb Z}}

\begin{document}
\title{Difference Sets with Few Character Values}
\author{Tao Feng    \and
        Sihuang Hu  \and
        Shuxing Li  \and\\
        Gennian Ge
}
\institute{
Tao Feng \and Sihuang Hu  \and Shuxing Li  \at
Department of Mathematics,
Zhejiang University,
Hangzhou, 310027, Zhejiang, China\\\\
Gennian Ge \at
School of Mathematical Sciences,
Capital Normal University,
Beijing, 100048, China\\\\
Tao Feng \at
\email{tfeng@zju.edu.cn}\\\\
Sihuang Hu \at
\email{husihuang@zju.edu.cn}\\\\
Shuxing Li \at
\email{sxli@zju.edu.cn}\\\\
Gennian Ge \at
\email{gnge@zju.edu.cn}
}

\date{Received: date / Accepted: date}

\maketitle

\begin{abstract}
The known families of difference sets can be subdivided into three classes:
difference sets with Singer parameters, cyclotomic difference sets,
and difference sets with gcd$(v,n)>1$. It is remarkable that all the known
difference sets with gcd$(v,n)>1$ have the so-called character divisibility
property. In 1997, Jungnickel and Schmidt posed the problem of constructing difference
sets with gcd$(v,n)>1$ that do not satisfy this property.
In an attempt to attack this problem, we use difference sets with three
nontrivial character values as candidates, and get some necessary
conditions.
\keywords{Association schemes \and Character divisible property \and Character values \and Difference sets}
 \subclass{MSC 05E30 \and MSC 05B10}
\end{abstract}

\section{Introduction}
\label{section:introduction}
Let $G$ be a finite abelian group of order $v$ and exponent $m$.
A subset $D$ of size $k$ in $G$ is called a
$(v, k,\lambda)$ difference set if each nonidentity element of $G$ can be represented
as $d_1d_2^{-1}$, $d_1, d_2 \in D$ in exactly $\lambda$ ways.
The order of $D$ is defined to be $n = k-\lambda$.
For a subset $A$ of $G$, we set $A^{(-1)}=\{g^{-1}\,|\, g\in A\}$;
also we use the same $A$ to denote the group ring element
$\sum_{g\in A}g\in \ZZ[G]$. Then, it is not hard to see that a
$k$-subset $D$ of a group $G$ of order $v$ is a $(v,k,\lambda)$ difference
set in $G$ if and only if it satisfies the following
equation in the group ring $\ZZ[G]$:
$$
DD^{(-1)}=n+\lambda G.
$$

Besides group rings, character theory is another very fruitful tool in the study of difference
sets. For a finite abelian group $G$,  we use $\widehat{G}$
to denote its character group, and $\chi_0$ the principal character.
The \textbf{Fourier inversion formula} below will be used frequently.

\begin{lemma}\label{inversion formula}
Let $G$ be an abelian group of order $v$. If $A=\sum_{g\in G}a_g
g\in \ZZ[G]$, then
$$
a_h=\frac{1}{v}\sum_{\chi\in\widehat G}\chi(A) \chi (h^{-1})
$$
for all $h\in G$, where $\chi(A)=\sum_{g\in A}a_g\chi(g)$.
\end{lemma}

One useful consequence of the inversion formula is as follows. Let $G$ be
an abelian group of finite order, and let $A$ and $B$ be two
elements of $\ZZ[G]$. Then $A=B$ if and only if
$\chi(A)=\chi(B)$ for all characters $\chi$ of $G$. The next
result is a standard characterization of difference sets by using
their character values.

\begin{proposition}
Let $G$ be an abelian group of order $v$ and $\chi \in \widehat{G}$.
Let $k$ and $\lambda$ be positive integers satisfying
$k(k-1)=\lambda (v-1)$. Then a $k$-subset $D$ of $G$ is a
$(v,k,\lambda)$ difference set in $G$ if and only if
$$
\chi(D)\overline{\chi(D)}= \left\{
                            \begin{array}{lll}
                            n, & \mbox{if} & \chi \ne \chi_0,\\
                            k^2, & \mbox{if} & \chi = \chi_0.
                            \end{array} \right.
$$
\end{proposition}

For more background on difference sets, the interested reader may refer to \cite{BJL,Pott}.

The known families of difference sets can be subdivided into three classes:
difference sets with Singer parameters, cyclotomic difference sets,
and difference sets with gcd$(v,n)>1$.
There are five known families of difference sets with gcd$(v,n)>1$,
namely Hadamard difference sets, the McFarland and the Spence family,
a series similar to Spence difference sets discovered by Davis and Jedwab \cite{DJed},
and a series generalizing Hadamard difference sets found by Chen \cite{Chen}.
We say a difference set $D$ has \textbf{character divisibility property} if
$\sqn\,|\,\chi(D)$ for each nonprincipal character $\chi$ of $G$.
It is  remarkable that all the known difference sets with gcd$(v,n)>1$
have this property. So it is natural to ask the following problem, which was posed by
Jungnickel and Schmidt in their survey paper \cite{JS}.\\

\noindent\textbf{Research Problem:} Construct difference sets with gcd$(v,n)>1$
that do not have the character divisibility property.\\

This current project makes an attempt to attack this problem.
For a difference set $D$, we define
$$
X=X(D)=\{ \chi(D)\mid \chi \in \widehat{G}, \chi \ne \chi_0\},
$$
which is the set of character values $\chi(D)$, where $\chi$ ranges
over all the nonprincipal characters of $G$. We will use difference
sets with $|X|=3$ as candidates and derive some necessary
conditions. With the aid of computer, some infinite classes of plausible
parameters satisfying all the necessary conditions has been found.
Besides, we find several classes of parameters which meet almost all
the necessary conditions. We list these parameters as a supportive
evidence for the existence of difference sets without character
divisibility property.

Related to this paper, there are  some similar results on graphs.
In \cite{BM1,BM2,BM3}, Bridges and Mena studied the multiplicative design,
which is related to a family of three eigenvalue graphs.
Ma \cite{Ma1} considered the subset of a group with few character values under
the name of polynomial addition sets.
Overall, these works are particular cases of the problem of graphs
with few eigenvalues, see \cite[Chapter 15]{BH}.

This paper is organized as follows. In Section~\ref{section:necessay conditions},
we obtain some necessary conditions for the existence of difference sets
with exactly three nontrivial character values.
Then, according to Lemma~\ref{lemma:only one prime divisor}, we split
our discussion into three cases, which are handled separately
in Sections~\ref{section:d=p}-\ref{section:d=-1}.
Another three special cases are considered in Section~\ref{section:special cases}.
A brief conclusion will be given in the last section.

\section{Necessary Conditions}\label{section:necessay conditions}
In the following, we will always assume that $\chi(D)$ takes
exactly three nontrivial character values, denoted by $a,b$ and $c$, when $\chi$
ranges over all the nonprincipal characters of $G$.

Here we fix some notation. Write $\ZZ_m^*$ for the multiplicative
group of units in $\ZZ_m$. For each $t \in \ZZ_m^*$, define
$\sigma_t\in \textup{Gal}(Q(\xi_m)/Q)$ by $\sigma_t(\xi_m)=\xi_m^t$,
and every element in $\textup{Gal}(Q(\xi_m)/Q)$ is of this form. For
each $\chi\in\widehat{G}$, define $\chi^t(x)=\sigma_t(\chi(x))$ for
all $x\in G$, which is also a character of $G$. For a subset $U$ of
$\widehat{G}$, set $U^{(t)}=\{\chi^t\,|\,\chi\in U\}$. We define
$$
U_z=\{\chi\in \widehat{G}\,\backslash\{\chi_0\}\,|\,\chi(D)=z\}
$$
for each $z$ in $\{a,b,c\}$. Then, $U_a$, $U_b$ and $U_c$ form a partition of
nonprincipal characters of $G$. For each character $\chi$ of $G$,
$\chi^{-1}$ is also a character of $G$, and $\chi^{-1}(D)=\sigma_{-1}(\chi(D))$,
therefore we have
$$
\{\sigma_{-1}(a),\sigma_{-1}(b),\sigma_{-1}(c)\}=\{a,b,c\}.
$$
Then, at least one element of $\{a,b,c\}$ is fixed by
$\sigma_{-1}$. Without loss of generality, we assume that
$c=\sigma_{-1}(c)$, i.e., $c$ is a real number. Then, we see that
$b=\sigma_{-1}(a)=\bar{a}$, and $U_a^{(-1)}=U_b$.

Let $\chi$ be a character in $U_c$, that is $\chi(D)=c$.
Together with $\chi(D)\overline{\chi(D)}=n$, we obtain $c=\pm\sqrt{n}$. By taking
$G\setminus D$ instead of $D$, which is also a difference set, we may assume that $c=\sqrt{n}$.
Clearly neither of $a,\bar{a}$ is equal to $\pm \sqrt{n}$.
Since $\textup{Gal}(Q(\xi_m)/Q)$ is abelian, we have
$$
\sigma_{-1}(\chi^t(D))=\sigma_{t}\sigma_{-1}(\chi(D))=\chi{}^t(D),
$$
for each $t\in \ZZ_m^*$. It now follows that
$\sigma_t(\sqrt{n})=\chi^t(D)=\sqrt{n}$ for each $t\in \ZZ_m^*$.
Hence the number $\sqrt{n}$ is an integer and $U_c^{(t)}=U_c$ for
each $t\in \ZZ_m^*$.

Similarly, it is easy to see that the subgroup $T:=\{t \in \ZZ_m^*|
\sigma_t(a)=a\}$ has index $2$ in $\ZZ_m^*$ and
$U_a^{(t)}=U_a$, $U_b^{(t)}=U_b$ for each $t\in T$.
Consequently, for each $\chi \in \widehat{G}$ and $t \in T$,
we have $\chi(D)=\chi^t(D)=\chi(D^{(t)})$. By the inversion
formula, we infer that $D$ is fixed by $T$, namely, $D^{(t)}=D$ for each $t \in T$.
On the other hand, we see that $Q(a)$ is a quadratic subfield of $Q(\xi_m)$,
by the fundamental theorem of Galois theory.
Then, $Q(a)=Q(\sqrt{d})$ for some squarefree integer $d$.
Now we recall some well-known results about quadratic and cyclotomic fields, which
can be found in any standard textbook on algebraic number theory, e.g. \cite{IrRo}.

The ring of algebraic integers of $Q(\sqrt{d})$ is $\ZZ[1,\beta]$ with
$$
\beta =\left\{\begin{array}{ll}
                \sqrt{d}, &\textup{if } d\equiv 2,3\pmod{4}, \\
                (-1+\sqrt{d})/2,&\textup{if } d\equiv 1\pmod{4},
              \end{array}\right.
$$
and the discriminant of $Q(\sqrt{d})$ is
$$
\Delta_d =\left\{\begin{array}{ll}
                4d, &\textup{if } d\equiv 2,3\pmod{4}, \\
                 d,&\textup{if } d\equiv 1\pmod{4}.
              \end{array}\right.
$$
The discriminant of $Q(\xi_m)$ is equal to
$$
(-1)^{\phi(m)/2}\frac{m^{\phi(m)}}{\prod_{p|m}p^{\phi(m)/(p-1)}},
$$
which has the same prime divisor with $m$; unless $m\equiv 2\pmod{4}$ in which case it has
the same odd prime divisor with $m$. Because a prime $p$ ramifies in a field if and only
if $p$ divides the discriminant of this field,
we see that each prime divisor of $\Delta_d$ is a divisor of $m$.
It follows that $d|m$, and $4|m$ when $\Delta_d$ is even.

For each prime $p\,|m$, we denote the Sylow $p$-subgroup of $G$ by $G_p$, and write
$G=G_p\times W,$
where $W$ is a subgroup of order $w$ and $w$ is coprime to $p$. We also
assume that $|G_p|=p^s$ for some integer $s$. Now we come to our first
result, which is about the discriminant of the quadratic field $Q(a)$.

\begin{lemma}\label{lemma:only one prime divisor}
The discriminant $\Delta_d$ of the quadratic field $Q(a)$ has only one prime divisor.
Moreover, we have
$$
\Delta_d =\left\{\begin{array}{lll}
                -p, &\textup{if } d=-p,\ \textup{where}\  p\  \textup{is a prime}\ \equiv 3\pmod 4, \\
                -8,&\textup{if } d=-2,\\
                -4,&\textup{if } d=-1,
              \end{array}\right.
$$
and the above three cases are the only possibilities.
\end{lemma}
\begin{proof}
Take any prime $p$ dividing $\Delta_d$, and assume that $p^s||v$. For each
nonprincipal character $\chi$ which is principal on $G_p$,
$\chi(D)$ is in the field $Q(\xi_{\textup{ord}(\chi)})$ whose discriminant is coprime to $p$.
It follows that $\chi(D)=\sqrt{n}$. Using the inversion
formula, we can check that the homomorphic image of $D$ in $\bar{G}=G/G_p$ is
\begin{equation}\label{hom_1}
 \bar{D}=\sqrt{n}+\frac{k-\sqrt{n}}{w}\bar{G},
\end{equation}
where $w=v/p^s$. It follows that
$$
w|(k-\sqrt{n}),
$$
equivalently, $v|p^s(k-\sqrt{n}).$ If there is another prime divisor
$q$ of $\Delta_d$, then $v|{q^r}(k-\sqrt{n})$ with $q^r||v$.
Because gcd$(p^s,q^r)=1$, $v|(k-\sqrt{n})$ which is false.
Therefore $\Delta_d$ has only one prime divisor.
It follows that either $\Delta_d=\pm p$
for an odd prime $p$ or $\Delta_d=\pm 2^r$ for some integer $r\geq
2$. Correspondingly, we have
\begin{itemize}
  \item[(a)] $d=p^*=(\frac{-1}{p})p$, or
  \item[(b)] $d\in\{-1,-2,2\}$.
\end{itemize}
In case (a), we have $p\equiv 3\pmod{4}$, since
$\sigma_{-1}(\sqrt{p^*})=(\frac{-1}{p})\sqrt{p^*}=-\sqrt{p^*}$.
In case (b), $d=2$ does not occur, since $\sqrt{2}=\xi_8+\xi_8^7$ is fixed
by $\sigma_{-1}$. Hence we are only left with the three cases listed in this lemma.
\qed
\end{proof}

\noindent\textbf{Remark:} When $p=2$, the assertion $v|2^s(k-\sqrt{n})$ in
the above proof can be improved.
Let $N$ be a subgroup of order $2^{s-t}$ such that
the Sylow 2-subgroup of $\bar{G}=G/N$ is elementary abelian. Let $rk_2(G)$ denote
the maximum possible integer for such $t$.
Then the same argument as above will give $\bar{D}=\sqn+\frac{k-\sqn}{v/2^{s-rk_2(G)}}\bar{G}$.
It now follows  $v|(2^{s-rk_2(G)}(k-\sqrt{n}))$.\\

When $a$ is a pure imaginary number, we have the following result.

\begin{lemma}\label{lemma:a is purely imaginary}
If $a+\baa=0$, then $d=-1.$
\end{lemma}
\begin{proof}
From $a+\baa=0$ and $a\baa=n$, it follows that $a=\pm i\sqn$.
So we obtain $Q(a)=Q(\sqrt{-1}),$ then $d=-1$.
\qed\end{proof}

For convenience, we introduce some other notations:
$\Delta = 2\sqrt{n}-a-\bar{a},$
$\Omega = (v(\sqrt{n}-a))/((a-\bar{a})\Delta),$
and
$R=(k-\sqrt{n})/\Delta.$
Write $D=\sum_{g\in G}d_gg$ and
$D^{(-1)}=\sum_{g\in G}d_g'g$, with each of $d_g,d_g'$
being $0$ or $1$. From the inversion formula, we have the
following equations for each element $g\in G$:
\begin{eqnarray*}
vd_g&=&ag^{-1}(U_a)+\bar{a}{g^{-1}(U_b)}+\sqrt{n}g^{-1}(U_c)+k;\\
vd_g'&=&\bar{a}g^{-1}(U_a)+a{g^{-1}(U_b)}+\sqrt{n}g^{-1}(U_c)+k;\\
v\delta_g&=& g^{-1}(U_a)+{g^{-1}(U_b)}+g^{-1}(U_c)+1.
\end{eqnarray*}
Here $g^{-1}(U_z)=\sum_{\chi\in U_z} \chi(g^{-1})$ for $z=a,b,c$;
$\delta_g=1$ if $g=1_G$ and $0$ otherwise. Then we get

\begin{eqnarray*}
 g^{-1}(U_a)&=&\frac{(vd_g-k+\bar{a})(\sqrt{n}-a)-(vd_g'-k+a)(\sqrt{n}-\bar{a})
                -\sqrt{n}\bar{a}v\delta_g+\sqrt{n}av\delta_g}{(a-\bar{a})\Delta};\\
 g^{-1}(U_b)&=&\overline{g^{-1}(U_a)};\\
 g^{-1}(U_c)&=&\frac{v(d_g+d_g')-(a+\bar{a})(v\delta_g-1)-2k}{\Delta}.
\end{eqnarray*}

Especially, when $g=1_G$, the above equations give
\begin{eqnarray*}
|U_a|&=&|U_b|=\frac{v(\sqrt{n}-d_1)}{\Delta}+R,\textup{ and}\\
|U_c|&=&\frac{v(2d_1-a-\bar{a})}{\Delta}-1-2R.
\end{eqnarray*}
When $g\neq 1_G$, we divide it into four cases depending on the
values of $d_g$, $d_g'$ as listed in Table~\ref{tab:1}.
%

\begin{table}
\caption{}
\label{tab:1}       
\begin{tabular}{ c c  c c  c}
\hline\noalign{\smallskip}
$d_g$   & $d_g'$  & $g^{-1}(U_a)$        & $g^{-1}(U_b)$          & $g^{-1}(U_c)$ \\
\noalign{\smallskip}\hline\noalign{\smallskip}
$1$     & $1$     &$-\frac{v}{\Delta}+R$ & $-\frac{v}{\Delta}+R$  & $\frac{2v}{\Delta}-1-2R$\\
\noalign{\smallskip}\noalign{\smallskip}
$1$     & $0$     & $\Omega+R$           & $\bar{\Omega}+R$       & $\frac{v}{\Delta}-1-2R$\\
\noalign{\smallskip}\noalign{\smallskip}
$0$     & $1$     & $\bar{\Omega}+R$     & $\Omega+R$             & $\frac{v}{\Delta}-1-2R$\\
\noalign{\smallskip}\noalign{\smallskip}
$0$     & $0$     & $R$                  & $R$                     & $-1-2R$  \\
\noalign{\smallskip}\hline
\end{tabular}
\end{table}

It is easy to see that $|D\cap D^{(-1)}|$ is just the coefficient of $1_G$ in $D^2$.
Using the inversion formula, we have
\[v|D\cap D^{(-1)}|=|U_a|(a^2+\bar{a}^2)+|U_c|n+k^2,\]
which implies
\begin{eqnarray*}
  |D\cap D^{(-1)}|&=&{|U_a|(a^2+\bar{a}^2)+|U_c|n+k^2\over v}\\
  &=&\frac{1}{v}\left[|U_a|(-2n+a^2+\bar{a}^2)+(2|U_a|+|U_c|)n+k^2\right]\\
  &=&\frac{1}{v}\left[-(v(\sqrt{n}-d_1)+k-\sqrt{n})(2\sqrt{n}+a+\bar{a})+(v-1)n+k^2\right]\\
  &=&k-(\sqrt{n}-d_1+\frac{k-\sqrt{n}}{v})(2\sqrt{n}+a+\bar{a})<k\\
\end{eqnarray*}

Next, we define the following sets, which form a partition of $G\setminus\{1_G\}$:
$E_1=D\cap D^{(-1)}\setminus\{1_G\}, E_2=D\setminus D^{(-1)},
E_3=D^{(-1)}\setminus D$, and $E_4=G\setminus(D\cup D^{(-1)}\cup\{1_G\}).$
Neither $E_2$ nor $E_3$ is empty, otherwise, $D=D^{(-1)}=D\cap D^{(-1)}$ implies
that $\chi(D)=\chi(D^{(-1)})=\overline{\chi(D)}$, which is false for $\chi\in U_a$.
Similarly, at least one of $E_1$, $E_4$ is not empty, or else $D+D^{(-1)}=G-1+2d_1$
and $\chi(D)$ takes only two character values when $\chi$
ranges over all nonprincipal characters. Therefore, at least three  of $E_i, 1\le i\le 4$ are not empty.

It is worthy to notice that when $E_1$ is empty but $E_4$ is not, we
obtain a $3$-class association scheme on
$\widehat{G}$. Suppose $\widehat{G}$ has conjugate classes $\{C_0,
C_1,\ldots,C_d\}$, where $C_0=\{ \chi_0 \}$. Define the $i$-th
relation $R_i$ by $(x,y) \in R_i$ if and only if $yx^{-1} \in C_i$.
It is well known that $\mathscr{X}=(\widehat{G},\{R_i\}_{0 \le i \le
d})$ is a $d$-class association scheme \cite{BI}.
Set $\overline{C_i}=\sum_{x \in C_i} x$ for $0 \le i \le d$. Then $\{\overline{C_0},\overline{C_1},\ldots,\overline{C_d}\}$ forms a Schur ring \cite{Ma2}.
With the
Bannai-Muzychuk criterion \cite{Ban,Mu} and the information provided
in Table \ref{tab:1}, we obtain a $3$-class fusion scheme of $\mathscr{X}$
whose first eigenmatrix is

$$
P=\quad\bordermatrix{%
& \chi_0 & U_a & U_b & U_c\cr
1 & 1 & \frac{v(n-d_1)}{\Delta}+R & \frac{v(n-d_1)}{\Delta}+R & \frac{v(2d_1-a-\bar{a})}{\Delta}-1-2R \cr
&&&& \cr
E_2 & 1 & \Omega+R & \bar{\Omega}+R & \frac{v}{\Delta}-1-2R \cr
&&&& \cr
E_3 & 1 & \bar{\Omega}+R & \Omega+R & \frac{v}{\Delta}-1-2R \cr
&&&& \cr
E_4 & 1 & R & R & -1-2R \cr
},
$$
with
$$
\det P=\frac{v^3}{(a-\bar{a})\Delta}.
$$

In the following, we will derive some necessary conditions from the above discussions.
We fix $p$ to be the only prime divisor of $\Delta_d$.
Write $\sqn=p^xu$ for some nonnegative integer $x$ and $(p,u)=1$, i.e., $p^x||\sqn$.

\begin{enumerate}
\item[(1)]
$\Delta|v$, $\Delta|(k-\sqn)$, $(a-\bar{a})\Delta|v(\sqrt{n}-a)$, $(a-\bar{a})|v$.\\
These follow from the fact that all entries in Table \ref{tab:1} are algebraic integers,
$E_2$ and $E_3$ are not empty, and at least one of $E_1$, $E_4$ is not empty.
The last one comes from $\Omega-\bar{\Omega}=v/({a-\bar{a}})$.\\
\item[(2)]
$v\,|(k-\sqrt{n})(2\sqrt{n}+a+\bar{a})$.\\
This is because $|D\cap D^{(-1)}|$ is an integer.\\
\item[(3)]
$w|(k-\sqrt{n}),$ $\sqrt{n}+\frac{k-\sqrt{n}}{w}\leq p^s$.\\
These come from the proof of Lemma~\ref{lemma:only one prime
divisor}. If $p=2$, we actually have $v|2^{s-rk_2(G)}(k-\sqn)$
and
$\sqn+\frac{k-\sqn}{2^{rk_2(G)}w}\leq 2^{s-rk_2(G)}$.\\
\item[(4)]
$p\,|(2\sqrt{n}+a+\bar{a})$.\\
Suppose not, then we otain $p^s|(k-\sqn)$ from (2).
Recall that $w|(k-\sqn)$, so we have $v|(k-\sqn)$ which is false.\\
\item[(5)]
$1\leq 2d_1-a-\baa$.\\
This comes from $g^{-1}(U_c)$ is an integer not exceeding $|U_c|$.\\
\end{enumerate}

Now, we give some information about other prime divisors of the order of $G$.

\begin{proposition}\label{proposition:q|a-baa}
If $q$ is another prime divisor of $v$, then we have $q|(a-\baa)$.
\end{proposition}

\begin{proof}
Assume that $q \nmid (a-\baa)$, then $q\nmid (\sqn-a)$, since otherwise $q\mid(\sqn-\baa)$,
and their difference gives $q \mid (a-\baa)$. Let
$$
\bar{G}=G_p\times\langle \alpha:\alpha^q=1\rangle
$$
be a quotient group of $G$, and define $\bar{D}$ as the image of $D$ in $\bar{G}$.

When $d=-p$ where $p$ is an odd prime, we have $\sigma_t(\sqrt{d})=\sigma_t(\sqrt{p^*})=(\frac{t}{p})\sqrt{p^*}$.
Thus $T=\{t\in\mathbb{Z}_m^*\vert ({t\over p})=1\}.$
When $d=-1$, we have $p=2$ and $4\vert m$. Since $\sqrt{-1}=\xi_4$, we see that
$T=\{t\in\mathbb{Z}_m^*\vert t\equiv 1\pmod 4\}.$
When $d=-2$, we have $p=2$ and $8\vert m$. Since $\sqrt{-2}=\xi_8+\xi_8^3$,
we see that $T=\{t\in\mathbb{Z}_m^*\vert t\equiv 1,3\pmod 8\}.$
So whenever $d=-p,-1$ or $-2$, we can find an integer $t\in T$
satisfying that $t\equiv 1\pmod{p^s}$ and $t\equiv i\pmod q$
for any integer $i,1\le i\le q-1$, by the Chinese remainder theorem.

Because $D$ is fixed by $T$, we can see that
$$\bar{D}=D_0+D_1(\langle \alpha\rangle-1)$$
with $D_0,D_1\in \ZZ[G_p]$. Denote $\langle \alpha:\alpha^q=1\rangle$ by $G_q$,
let $\chi_1 \times \chi_2$ be a character of $G_p \times G_q$.
When $\chi_2$ is principal on $G_q$, we have
\begin{eqnarray}
\chi_1(D_0)+(q-1)\chi_1(D_1)=c_1,
\end{eqnarray}
otherwise,
\begin{eqnarray}
\chi_1(D_0)-\chi_1(D_1)=c_2.
\end{eqnarray}
$c_1,c_2\in\{a,\baa,\sqn\}$ if $\chi_1$ is nonprincipal on $G_p$, and $c_1=k,c_2=\sqn$
otherwise. When $\chi_1$ is nonprincipal, by taking the difference of the above
two equations, we obtain $q\chi_1(D_1)=c_1-c_2$, and our assumption forces
$c_1=c_2$ and $\chi_1(D_1)=0$. When $\chi_1$ is principal, we have
$\chi_1(D_1)=(k-\sqn)/{q}$. Hence
$$
D_1=\frac{k-\sqn}{qp^s}G_p,
$$
which gives $p^s|(k-\sqn)$. Together with $w|(k-\sqn)$, we obtain $v|(k-\sqn)$
which is false. Therefore $q|(a-\baa)$.
\qed\end{proof}


In addition, we obtain some results about the multiplier of $D$.

\begin{proposition}\label{proposition:D^(q)=D}
If there is a prime $q$ dividing $n$ and coprime to $v$, then we have $D^{(q)}=D$.
\end{proposition}

\begin{proof}
Let $Q$ be a prime ideal in $\ZZ[\xi_v]$ lying over $q$, then $Q|n$ since $q|n$.
If $Q|a$ and $Q|\baa$, then $Q|\Delta=2\sqn-a-\baa$, which gives $Q|v$ since
$\Delta|v$. This contradicts to the fact that $q$ is coprime to $v$.
For $n=a\baa$,  we have $Q$ divides exactly one of $a,\baa$. Since
$\sigma_q(Q)=Q$,
$\{\sigma_q(a),\sigma_q(\baa)\}=\{a,\baa\}$,
we must have
$\sigma_q(a)=a$ and $\sigma_q(\baa)=\baa$. Because $\chi(D^{(q)})=\sigma_q(\chi(D))$,
and $\chi(D)$ only takes the values $k,\sqn,a,$ and $\baa$, we immediately get that
$\chi(D^{(q)})=\chi(D)$ when $\chi$ ranges over all characters. It follows from
the inversion formula that $D^{(q)}=D$.
\qed\end{proof}

\noindent\textbf{Remark:} (1) For such a prime $q$ as in
Proposition~\ref{proposition:D^(q)=D}, from $\sigma_q(a)=a$, namely
$\sigma_q(\sqrt{d})=\sqrt{d}$, we have: $(\frac{q}{p})=1$ if $d=-p$
with $p$ odd;
 $q\equiv 1,3\pmod{8}$ if $d=-2$;
 $q\equiv 1 \pmod{4}$ if $d=-1$.

(2) Let $q\neq p$ be a prime divisor of both $n$ and $v$. Then by
Proposition~\ref{proposition:q|a-baa}, we have $q|(a-\baa)$.
From $4n=(a+\baa)^2-(a-\baa)^2$, we get $q|(a+\baa)$.
If $q$ is odd, then $q| a$ and $q|\baa$.

\section{The case $d=-p$}\label{section:d=p}
We will first deal with the case $d=-p$ with $p$ an odd prime in this section.
Assume $d=-p$, where $p$ is a prime $\equiv 3\pmod 4$.

From Lemma~\ref{lemma:a is purely imaginary}, we know $a+\baa\neq
0.$ Some new inequalities and necessary conditions about
divisibility can be obtained under this assumption. We will use
notation like (6') to mark the property which is obtained under the
additional assumption that $p$ is an odd prime.
For an integer $l$, the greatest integer $z$ such that
$p^z$ divides $l$ is denoted by $\orp(l)$.

\begin{enumerate}
\item[(6')]
$p^x||(a+\baa)$, $p^x|(a-\baa)$, $p^x|k$, $s\geq x+1$.\\
Since $d=-p\equiv 1\pmod 4$ and $a$ is an algebraic integer in
$\mathbb{Q}(\sqrt{-p})$, there exist integers $e$ and $f$ such
that $a=e+f{-1+\sqrt{-p}\over 2}.$ Then $a-\bar{a}=f\sqrt{-p}$,
so $\orp((a-\baa)^2)$ is odd.
The fact $p^x||(a+\baa)$ follows from $4n=(a+\baa)^2-(a-\baa)^2$ and
$\orp((a-\baa)^2)$ is odd. Then we obtain $p^x|(a-\baa)$ and
$p^x|\Delta=2\sqn-a-\baa$. Since $\Delta|(k-\sqn)$, we have
$p^x|(k-\sqn)$ and hence $p^x|k$. From $ w|(k-\sqn)$, we have
$v|(k-\sqn)p^{s-x}$,
which gives $s\geq x+1$.\\
\item[(7')]
$x\geq 1$.\\
Take a character $\chi\in U_a$, and let $\chi'$ be another character
of $G$ which coincides with $\chi$ on $W$ and is principal on $G_p$.
Thus we see $\chi(D)=a$ and $\chi'(D)=\sqrt{n}$. Since
$(1-\xi_{p^s})|(\chi(D)-\chi'(D))$, it follows that
$(1-\xi_{p^s})|(a-\sqn)$. Similarly, we have
$(1-\xi_{p^s})|(\baa-\sqn)$. Recall that $p|(2\sqrt{n}+a+\bar{a})$,
we have
$(1-\xi_{p^s})|4\sqn$. Therefore we get $p|4\sqn$, which implies $x\geq 1$.\\
\item[(8')]
$p^x||k$, $p^x||(k-\sqn)$, $p^s||(k+\sqn)$, $\Delta|p^xw$.\\
From $k(k-1)=\lambda(v-1)$ and (1), we have
$\orp(\lambda)=\orp(k)\geq x$. Then from $k^2=n+\lambda v$ and
$\orp(n)=2x<x+s\leq\orp(\lambda v)$, we get $\orp(k)=x$, i.e.,
$p^x||k$. Because $(k+\sqn)-(k-\sqn)=2\sqn$, we readily verify that
at least one of $\orp(k+\sqn)$, $\orp(k-\sqn)$ is $x$. Now
$k^2-n=(k+\sqn)(k-\sqn)=\lambda v$ gives $\{\orp(k+\sqn),
\orp(k-\sqn)\}=\{s,x\}$.
 If $\orp(k-\sqn)=s$, together with $w|(k-\sqn)$ we get $v|(k-\sqn)$
which is false. Hence we have $\orp(k+\sqn)=s$, $\orp(k-\sqn)=x$.
The last one $\Delta|p^xw$ follows from $\Delta|v$, $\Delta|(k-\sqn)$, and gcd$(v,(k-\sqn))=p^xw$.\\
\end{enumerate}

Now define
$$
\gamma=\frac{k-\sqn}{p^xw}.
$$
From $w|(k-\sqn)$ and $p^x||(k-\sqn),$ we see that $\gamma$ is an integer coprime to $p$. From
$$
(k+\sqn)(k-\sqn)=k^2-n=\lambda v,
$$
we have
$$
\gamma(\lambda +n+\sqn)=p^{s-x}\lambda,
$$
which gives
\begin{eqnarray*}
\lambda&=&\frac{(n+\sqn)\gamma}{p^{s-x}-\gamma}=\frac{(p^xu+1)u}{p^{s-x}-\gamma}p^x\gamma,\\
k&=&n+\lambda=\frac{p^{s-x}n+\sqn\gamma}{p^{s-x}-\gamma}
  =\frac{(p^su+\gamma)u}{p^{s-x}-\gamma}p^x,\textup{and}\\
 w&=&\frac{k-\sqn}{p^x\gamma}=\frac{(n-\sqn)p^{s-x}+2\sqn\gamma}{\gamma p^x(p^{s-x}-\gamma)}
   =\frac{(p^xu-1)u}{\gamma}+\frac{(p^xu+1)u}{p^{s-x}-\gamma}.
\end{eqnarray*}
Write
\begin{eqnarray*}
        a+\baa&=&-p^x\alpha,\\
        a-\baa&=&\eta p^x\sqrt{-p}
\end{eqnarray*}
with $\alpha,\eta\in \ZZ$ and gcd$(p,\alpha)=1$. Because $1\leq
2d_1-a-\baa$ and $a+\baa\neq0$, we must have $\alpha\geq 1$. From
$4n=(a+\baa)^2-(a-\baa)^2$,  $4u^2=\alpha^2+p\eta^2$. From
Proposition~\ref{proposition:q|a-baa}, we see that
$\pi(w)=\pi(\eta)\setminus\{p\}$, where $\pi(w)$ denotes the set of prime divisors of $w$.
From $\sqrt{n}+{k-\sqrt{n}\over w}\le p^s$ we get $u+\gamma\leq p^{s-x}.$
After simplification, we reduce the above conditions to the following
list:

\begin{center}
\begin{tabular}{rrrr}
$\gamma|(p^xu-1)u,$ & $(p^{s-x}-\gamma)|(p^xu+1)u,$ & $2u+\alpha|w,$ & $ p^{s-x}|p^x(2u-\alpha)$,\\
$\eta|p^{s-x}w,$ & $u+\gamma\leq p^{s-x},$ & $\pi(w)=\pi(\eta)\setminus\{p\},$ &$4u^2=\alpha^2+p\eta^2$,\\
$s\geq x+1,$&$x\geq 1,$ & $\alpha\geq 1.$\\[5pt]
\end{tabular}
\end{center}

\noindent\textbf{Remark}:  (1) From the expression of $w$, we can show that
roughly $v\geq 4n$ with $\gamma$ as a variable in $(0,p^{s-x})$, and
minimized when $\gamma=\frac{p^{s-x}}{2}$. We have made no use of divisibility
conditions involving $a,\baa$. Notice that $w>1$, since $(2u+\alpha)|w$ and $2u+\alpha\ge 3$,
i.e., $G$ can not be a $p$-group.

(2) From $|U_c|\geq 0$, we get $v\geq\frac{2k-2d_1}{2d_1+p^x\alpha}+1$, which is trivial.

(3) Let $\widetilde{U_c}=\{\chi\in U_c\mid \chi\textup{ is nonprincipal on }G_p \}$.
Then we have $|\widetilde{U_c}|=|U_c|-(|W|-1)$. We define a group action of $\ZZ_{p^s}^{*}$ on
$\widetilde{U_c}$ by $(t,\chi)\rightarrow \chi^t$, where $t\in \ZZ_{p^s}^{*},\chi\in\widetilde{U_c}$.
It is not hard to see that each orbit has length divisible by $p-1$. Therefore we obtain
$(p-1)\big|(w(d_1-\sqn)-(k-\sqn))$.

(4) Let $T_1=\{t\in Z_{p^s}^*|(\frac{t}{p})=1\}$. We can define a similar group action
of $T_1$ on $U_a$ as above. By an analogous argument,
we have $\frac{p-1}{2}\big|(w\sqn-d_1+(k-\sqn))$.\\

At this point, we may speculate all the known parameter sets.
We have conducted a computer search for
$p\in\{3,5,7,11,13,17,19\},1\le x,s\le10,1\le \alpha,\eta\le 10^4$,
and failed to
find a parameter set satisfying all the conditions
listed in this section and Section~\ref{section:necessay conditions}. The examples below indicate that
such parameter sets might exist.

\begin{example}
Take $p=7$, $\alpha=8$, $\eta=24$, $u=32$, $\gamma=4$, $v=2^3\cdot 3^5\cdot 7^3$, $k=54656$, $n=2^{10}\cdot 7^2$. In this case, $3|w$, $3|(k-\sqn)$,  and with $d_1=0$, all the conditions in this section are satisfied except $(\sqn-d_1+\frac{k-\sqrt{n}}{v})(2\sqrt{n}+a+\bar{a})\leq k$ .
\end{example}

\begin{example}
Take $p=11$, $\alpha=30$, $\eta=48$, $u=81$, $\gamma=980$, $v=2^{10}\cdot 3\cdot 11^5$, $k=364287561$, $n=3^8\cdot11^4$. In this case, $w\equiv 2\pmod{5}$, $\sqn\equiv 1\pmod{5}$, $5|(k-\sqn)$, and with $d_1=1$, all the conditions in this section are satisfied except $\frac{p-1}{2}|(w\sqn-d_1+(k-\sqn))$.
\end{example}

\section{The case $d=-2$}
\label{section:d=-2}
Secondly, we  move to the case $d=-2$ and $\Delta_d=-8$. Under such assumption,
we have already known that $4|m$ in Section~\ref{section:necessay conditions}.
Define
$$
l=\textup{min}\{|N|\,\big{|}\, G/N \textup{ has exponent strictly divisible by } 4\}.
$$
So we may choose such a subgroup $N$ of $G$, whose order is $l$ and $G/N$ has exponent
strictly divisible by $4$. Considering the homomorphic image of $D$ in $G/N$,
as in the proof of Lemma~\ref{lemma:only one prime divisor}, we obtain
$$
\sqn+\frac{k-\sqn}{v/l}\leq l.
$$
Clearly $l$ is a power of $2$ and $l\leq 2^{s-2}$. Since $\sqrt{-2}$
lies in $Q(\xi_8)$ but not in $Q(\xi_4)$, we must have $8|\textup{exp}(G)$, hence $l\geq 2$.
Write $k=\sqn+\gamma\Delta$ for
some positive integer $\gamma$, and $a=u_1+u_2\sqrt{-2}$, where $u_1,u_2$
are two integers satisfying $u_1^2+2u_2^2=n$. We have run a computer search for
$-10^4\le u_1\le 1,1\le u_2,\gamma\le 10^4$ and found many parameter sets
such that all the conditions in Section~\ref{section:necessay conditions}
are satisfied. Here we only give
the following examples.

\begin{example}\label{example:4.1}
When $a=96(-1+2\sqrt{-2})t$, $\gamma=216t$, $n=2^{10}\cdot 3^4\cdot t^2$, $v=4n$,
with $t\in \{2^i,12\cdot 2^i,20\cdot 2^i\,|\,\textup{$i$ is a nonnegative integer}\}$,
all the conditions in Section~$\ref{section:necessay conditions}$ are satisfied.
If $t=1$, we get $l \ge {1\over (1-{k-\sqrt{n}\over v})}\sqrt{n}=2\sqn=2^6\cdot 3^2$, so $l\geq 2^{10}$.
This forces the Sylow $2$-subgroup of $G$ to be cyclic, but a result of Turyn
says such difference sets do not exist, cf.~\cite[Theorem 2.4.11]{Pott}.
\end{example}

\begin{example}
When $a=192(-7+4\sqrt{-2})t$, $\gamma=972t$, $n=2^{12}\cdot 3^6\cdot t^2$, $v=4n$,
$k=2n+\sqn$, with $t=2^i$ for some nonnegative integer $i$, all the conditions in
Section~$\ref{section:necessay conditions}$ are satisfied.
If $t=1$, $l\geq 2\sqn=2^7\cdot 3^3$, then $l\geq 2^{12}$. This is ruled out similarly as above.
\end{example}

\section{The case $d=-1$}
\label{section:d=-1}
The analysis of the case $d=-1$ and $\Delta_d=-4$ is almost the same
as that of the case $d=-2$. First define
$$
l=\textup{min}\{|N|\,\big{|}\, G/N\textup{ has exponent strictly divisible by } 2\}.
$$
Similarly as above, we have
\begin{equation}\label{equation: ineq of l}
\sqn+\frac{k-\sqn}{v/l}\leq l,
\end{equation}
where $l$ is a power of $2$, and $2\leq l\leq 2^{s-1}$.
Write $k=\sqn+\gamma\Delta$ for some positive integer $\gamma$,
and $a=u_1+u_2\sqrt{-1}$, where $u_1,u_2$ are two integers satisfying $u_1^2+u_2^2=n$.
We have also run a computer search for
$-10^4\le u_1\le 1,1\le u_2,\gamma\le 10^4$  and found many parameter sets
such that all the conditions in Section~\ref{section:necessay conditions}
are satisfied. Here we give an example.

\begin{example}\label{example:5.1}
When $a=160(-3+4\sqrt{-1})t$, $\gamma=500t$, $n=2^{10}\cdot 5^4\cdot t^2$, $v=4n$,
$k=2n+\sqn$, with $t=2^i$ for some nonnegative integer $i$,  all the conditions in
Section~$\ref{section:necessay conditions}$ are satisfied.
\end{example}

\section{Special Cases}\label{section:special cases}
Three special cases will be considered in this section:
\begin{enumerate}[(1)]
\item $D$ is a Hadamard difference set with $a+\bar{a}=0$;
\item $G$ is a $p$-group;
\item $U_c\cup\{\chi_0\}$ is a subgroup of $\widehat{G}$.
\end{enumerate}

\subsection{$D$ is a Hadamard difference set with $a+\bar{a}=0$}
Now let $D$ be a Hadamard difference set with three nontrivial character
values $\sqn$, $a$ and $\bar{a}$. Assume that $a+\baa=0,$
then by Lemma~\ref{lemma:a is purely imaginary}, we obtain $d=-1$ and $a=\pm i\sqn$.
Since $1\leq 2d_1-a-\baa$, we infer that $d_1=1,$ i.e., $1_G\in D$.
In the following, we split the discussion into two cases, according to
the parameter of $D$.

If $D$ is with parameter $(v,k,\lambda)=(4n,2n+\sqn,n+\sqn)$,
then it will satisfy all the necessary conditions,
with the only possible exception of (\ref{equation: ineq of l}),
which becomes $l\geq 2\sqn$ here.
Using formulas of Section~\ref{section:necessay conditions}, we have
$\Delta=2\sqn, R=\sqn,
 \Omega=-\sqn-i\sqn,$ and $|D\cap D^{(-1)}|=2\sqn$.
Denote $H=D+D^{(-1)}-G.$ There are $2\sqn (=|D\cap D^{(-1)}|)$ elements
whose coefficients are $1$ in $H$, and others have coefficients $0$ or $-1$.
In addition, the sum of all coefficients of $H$ is $2\sqn(=2k-v)$.
It now follows that $H$ is a subset of $G$.
Denote $M=U_c\cup\{\chi_0\}$, then we have
$$
\chi(H)=\chi(D)+\chi(D^{(-1)})-\chi(G)=\left\{
            \begin{array}{ll}
            2\sqn,\quad &\textup{if }\chi\in M;\\
            0,\quad &\textup{if }\chi\in\widehat{G}\setminus M.
            \end{array}
        \right.
$$
Using inversion formula, we find $H^2=2\sqrt{n}H$ and $H=H^{(-1)}$, i.e., $H$ is a
subgroup of $G$ with order $2\sqn$.
We say a character of $G$ annihilates a subgroup $H$ of $G$ if $\chi(h)=1$
for all $h\in H$. The set of all characters of $G$ annihilating $H$ is called
the annihilator of $H$ in $\widehat{G}$, and denoted by $H^{\perp}$.
We readily verify that $M=H^{\perp}$ and $|M|=|H^{\perp}|=|G/H|=2\sqn.$
As in the proof of Lemma~\ref{lemma:only one prime divisor}, we find $G_2^{\perp}\le M$
and $|G_2^{\perp}|=|G/G_2|=w$.
This leads to $w\,|2\sqn$, which implies $w^2\,|4n$.
We recall that $v=4n=2^sw$, where gcd$(w,2)=1$.
Hence  we find $w^2\,|2^sw$, which implies $w=1$, i.e., $G$ is a $2$-group.
It now follows that $n=2^{s-2}$, and $s$ must be even, since $\sqn$ is an integer.
Write $s=2m$ for some nonnegative integer $m$,
then $D$ is with parameter $(v,k,\lambda)=(2^{2m},2^{2m-1}+2^{m-1},2^{2m-2}+2^{m-1})$.
In this case, we may update Table \ref{tab:1} to get Table \ref{tab:2}.\\
%
\begin{table}
\caption{}
\label{tab:2}       
\begin{tabular}{ c  c c c c}
\hline\noalign{\smallskip}
$d_g$ & $d_g'$  & $g^{-1}(U_a)$ & $g^{-1}(U_b)$ & $g^{-1}(U_c)$\\
\noalign{\smallskip}\hline\noalign{\smallskip}
    $1$   & $1$     & $-\sqn$        &   $-\sqn$      &    $2\sqn-1$ \\ \noalign{\smallskip}
    $1$   & $0$     & $-i\sqn$      &   $i\sqn$     &     $-1$ \\\noalign{\smallskip}
    $0$   & $1$     & $i\sqn$       &   $-i\sqn$    &     $-1$ \\\noalign{\smallskip}
\noalign{\smallskip}\hline
\end{tabular}
\end{table}

On the other hand, if $D$ is with parameter $(v,k,\lambda)=(4n,2n-\sqn,n-\sqn)$,
then it can be ruled out. First we find $|D\cap D^{(-1)}|=1$ and $|U_c|=1+2\sqn$.
Denote $H=G+1-D-D^{(-1)}.$ Clearly $H$ is a subset of $G$ with size $1+2\sqn$.
We readily verify
$$
\chi(H)=\left\{
\begin{array}{ll}
            1+2\sqn, &\textup{if }\chi=\chi_0,\\
            1-2\sqn, &\textup{if }\chi\in U_c,\\
            1, &\textup{if }\chi\in\widehat{G}\setminus M,
        \end{array}
        \right.
$$
and
$$
g^{-1}(U_c)=\left\{
\begin{array}{ll}
  1+2\sqrt{n},& \textup{if }g= 1_G,\\
  1-2\sqrt{n},& \textup{if }g\in H,\\
  1,          & \textup{if }g\in G\backslash H,g\neq 1_G.
    \end{array}
    \right.
$$
By the inversion formula, we obtain
$$
H^2=(2\sqn-1)+(2-2\sqn)H+2G,
$$
and
$$
U_c^2=(2\sqn-1)+(2-2\sqn)U_c+2\widehat{G}.
$$
Because the coefficients of $H^2$ are nonnegative, we have $2-2\sqn+2\geq 0$, which gives $\sqn\leq 2$.
If $n=1$, we have $D=\{1_G\}$, for $|D|=k=1$ and $D$ contains the
identity element. But this contradicts to the fact that $\chi(D)$ takes three values.
If $n=4$, then we get $v=16$ and $U_c^2=3+2(\widehat{G}-U_c)$.
Hence for every element of $U_c$, its coefficient in $U_c^2$ is zero.
By the result of Turyn \cite[Theorem 2.4.11]{Pott},
$G$ cannot be cyclic, so must be isomorphic to $\ZZ_8\bigoplus \ZZ_2$ or $\ZZ_4\bigoplus\ZZ_4$.
It is a well known fact that $\widehat{G}\cong G$.
In each case, there are exactly three elements of order $2$ in $\widehat{G}$, and they all belong to $U_c$.
Hence at least one element of order $2$ must has positive coefficient in $U_c^2$. This is a contradiction.

The following is a summary of our discussions.

\begin{lemma}\label{hadamardDS}
Let $D$ be a Hadamard difference set of order n in an abelian group $G$ with three nontrivial character
values $\sqn$, $a$ and $\bar{a}$. If $a+\baa=0$, then $D$ must be
with parameter $(v,k,\lambda)=(2^{2m},2^{2m-1}+2^{m-1},2^{2m-2}+2^{m-1})$
for some nonnegative integer $m$. In particular, $G$ is a $2$-group.
Let $H=D+D^{(-1)}-G$. Then one has that $H$ is a subgroup of $G$, and
$H^{\perp}=\chi_0\cup\{\chi\in G | \chi(D)=\sqn\}.$
\end{lemma}

\subsection{$G$ is a $p$-group}
Recall that $w>1$ if $p$ is odd, hence we must have $p=2$.
By a result of Menon \cite{Menon}, the plausible difference set $D$ is with
parameter $(v,k,\lambda)=(4n,2n\pm\sqn,n\pm\sqn)$. Write $v=4n=2^s$ for some nonnegative integer $s$.
Since $\Delta|v$, we can assume $\Delta=2\sqn-a-\baa=2^u$
with $u$ being a nonnegative integer.
From $\Delta=2\sqn-a-\baa<4\sqn$ and $-a-\baa\geq 1-2d_1\geq -1$,
we get $2^{s/2}-1\leq 2^u<2^{s/2+1}.$
Consequently we find $2^u=2^{s/2}=2\sqrt{n}$, implying $a+\baa=0$.
This implies that $D$ satisfies the conditions of Lemma~\ref{hadamardDS}.
Hence $D$ must be
with parameter $(v,k,\lambda)=(2^{2m},2^{2m-1}+2^{m-1},2^{2m-2}+2^{m-1})$
for some nonnegative integer $m$. Let $H=D+D^{(-1)}-G$. Then one has that
$H$ is a subgroup of $G$, and $H^{\perp}=\{\chi_0\}\cup\{\chi\in \widehat{G} | \chi(D)=\sqn\}.$

Let us further assume that the exponent of $G$ is $4$. From
$$
\sqn+\frac{k-\sqn}{2^{rk_2(G)}}\leq 2^{s-rk_2(G)},
$$
we get $m\geq rk_2(G)$. By the definition of $rk_2(G)$, we can verify that
$m=rk_2(G)$, and $G\cong \ZZ_4^{m}$. Since $a=\pm i\sqn$, we find that
each $\chi \in U_a\cup U_a^{(-1)}$ has order $4$. Notice that
$\widehat{G}\cong \ZZ_4^{m}$ has $2^{m}-1=2\sqn-1$ elements of order $2$
and $|U_c|=2\sqn-1$, therefore $U_c$ contains exactly all the characters of order $2$.
In other words, $H^{\perp}=\{\chi_0\}\cup\{\chi\in \widehat{G} | \chi(D)=\sqn\}$ is the unique
maximal elementary abelian subgroup in $\widehat{G}$.
Such difference sets have been constructed by Davis and Polhill \cite{DP}.\\

\subsection{$U_c\cup\{\chi_0\}$ is a subgroup of $\widehat{G}$}
Let $M=U_c\cup\{\chi_0\}.$
Since $M$ is a subgroup, we see that $\psi(M)$ can only take two character values $0$ and $|M|$,
for each nonprincipal character $\psi$ of $\widehat{G}$. Together with the fact
$$
\psi(M)\in\bigg\{-2R,\frac{v}{\Delta}-2R, \frac{2v}{\Delta}-2R\bigg\}
$$
as listed in Table \ref{tab:1}, we readily verify
$$
\frac{v}{\Delta}-2R=0,\quad \frac{2v}{\Delta}-2R=|M|,
$$
since $-2R<0$ and $\frac{v}{\Delta}-2R<\frac{2v}{\Delta}-2R$.
We recall that $R=\frac{k-\sqrt{n}}{\Delta}$, therefore $v=2(k-\sqn)$ and $|M|=\frac{v}{\Delta}$.
On the other hand, we have $k^2=n+(k-n)v$. It now follows $v=4n$, $k=2n+\sqn$ and $\lambda=n+\sqn$,
i.e., $D$ is with parameter $(v,k,\lambda)=(4n,2n+\sqn,n+\sqn)$.
Let $H=D+D^{(-1)}-G$. Then we have
$$
\chi(H)=\left\{
        \begin{array}{cc}
            2\sqn,& \textup{if } \chi\in M;\\
            a+\baa,& \textup{if } \chi\in\widehat{G}\setminus M.
        \end{array}
        \right.
$$
From the inversion formula, we find $H=a+\baa+M^{\perp}.$
Comparing the coefficients of $1_G$ on each side, it now follows
$$
a+\baa=\left\{
        \begin{array}{cc}
                0, & \textup{if } d_1=1;\\
                -2,& \textup{if } d_1=0.
        \end{array}
        \right.
$$
If $a+\baa=-2,$ then we have $\Delta=2\sqn-a-\baa=2\sqn+2$. From $\Delta|(k-\sqn)$, we have
$(\sqn+1)|n$, which is false. Hence $a+\baa=0$.
Therefore $D$ also satisfies the conditions of Lemma~\ref{hadamardDS}.

Now we see that in either of the above two special cases,
$D$ satisfies the conditions of Lemma~\ref{hadamardDS}.
We state the following as a summary of this section.

\begin{theorem}\label{threeCasesOneResult}
Let $D$ be a difference set in an abelian group $G$ with three nontrivial character
values $\sqn$, $a$ and $\bar{a}$. Denote $M=\{\chi_0\}\cup\{\chi\in\widehat{G}\,|\,\chi(D)=\sqn\}$.
If
\begin{enumerate}
\item[(1)] $D$ is a Hadamard difference set with $a+\bar{a}=0$, or
\item[(2)] $G$ is a $p$-group, or
\item[(3)] $M$ is a subgroup of $\widehat{G}$,
\end{enumerate}
then $D$ is with parameter $(v,k,\lambda)=(2^{2m},2^{2m-1}+2^{m-1},2^{2m-2}+2^{m-1})$
for some nonnegative integer $m$. In particular, $G$ is a $2$-group.
Let $H=D+D^{(-1)}-G$. Then one has that $H$ is a subgroup of $G$ and $H^{\perp}=M.$
\end{theorem}

\section{Conclusion}\label{section:conclusion}

In this paper, we make an attempt to find difference sets without the character divisibility property. Under the assumption that the difference sets have only three distinct nontrivial character values, we have derived some restrictions on the parameters. It turns out that their character values all lie in the field $Q(\sqrt{-d})$, where $d=1$, $d=2$ or $d$ is an odd  prime congruent to $3$ modulo $4$.  We have conducted a computer search for plausible parameters satisfying all these conditions. When $d=1$ or $2$, we have found some plausible parameter sets satisfying all our conditions, some of which are listed in  Examples~\ref{example:4.1}-\ref{example:5.1}. It is an interesting open problem to rule out these parameter sets or construct a difference set with these parameters. An affirmative answer to this problem will provide difference sets
without the character divisibility property. On the other hand, we have not found any parameter
set satisfying all the derived conditions when  $d$ is an odd prime. 

At last, we take an initial step towards
the case of $|X|=4$. Similarly we introduce the following
sets
$$
U_z =\{ \chi\in \widehat{G} \,|\, \chi(D)=z\}
$$
for each $z\in\{a,b,c,d\}$. These four sets will form a partition of $\widehat{G}\setminus\{\chi_0\}$.
According to the lengths of orbits of $Gal(Q(\xi_m)\slash Q)$ acting on
$a,b,c,d$, we may divide it into the following four cases:
\begin{enumerate}[(1)]
\item
The Galois group $Gal(Q(\xi_m)\slash Q)$ acts transitively. For each prime $p\,|v$,
there exists a character $\chi$ with ord$(\chi)=p$, so $\chi(D)\in Q(\xi_p)$. From the assumption that
$Gal(Q(\xi_m)\slash Q)$ acts transitively on the set $\{a,b,c,d\}$, we have
$\{a,b,c,d\}\subseteq Q(\xi_p)$. If $v$ has another prime divisor $q$ different from $p$,
then there exists a character $\psi$ such that ord$(\psi)=q$, it follows as above that
$\{a,b,c,d\}\subseteq Q(\xi_q)$. Since $Q(\xi_p)\cap Q(\xi_q)=Q$, we have
$\{a,b,c,d\}\subseteq Q$, which is a contradiction. Therefore, the order
of $G$ must be a prime power, i.e., $G$ is $p$-group.\\

\item
There is one fixed point $\{a\}$, and an orbit of length three $\{b,c,d\}$.
Then $a\in Q$. From $a\baa=n$, we have $a=\sqrt{n}$ or $-\sqrt{n}$. Since
$\{\sigma_{-1}(b),\sigma_{-1}(c),\sigma_{-1}(d)\}=\{b,c,d\}$,
we may assume that $c=\bar{b}$ without loss of generality.
Then $\sigma_{-1}(d)=d$, that is $d$ is a real number.
In addition, $d\bar{d}=n$, therefore $d= -a$.
Then $d$ will also be fixed by $Gal(Q(\xi_m)\slash Q)$, contradicting to our assumption.\\

\item
 There are two fixed points $\{a\}$, $\{b\}$ and an orbit of length two $\{c,d\}$. By
the similar analysis as in Case (2), we have $a=\pm\sqrt{n},$ $b=-a,$ and
$d=\bar{c}$.\\

\item
 There are two orbits of length two $\{a,b\}$ and $\{c,d\}$. Then $b=\bar{a}$ and
$d=\bar{c}$.\\
\end{enumerate}

We leave this problem for our future considerations.

\begin{acknowledgements}
The authors are grateful to the anonymous reviewers for their detailed suggestions
and comments that improved the presentation and quality of this paper.
T. Feng was supported in part by Fundamental Research Fund for the Central Universities
of China, Zhejiang Provincial Natural Science Foundation under Grant
LQ12A01019, in part by the National Natural Science Foundation of China
under Grant 11201418, and in part by the Research Fund for Doctoral Programs
from the Ministry of Education of China under Grant 20120101120089.
S. Hu was supported by the Scholarship Award for Excellent Doctoral Student
granted by Ministry of Education.
G. Ge was supported by the National Natural Science Foundation of China under
Grant 61171198.
\end{acknowledgements}


%
%

\end{document}